\documentclass[12pt]{article}

\usepackage[inner=32mm,outer=31mm,tmargin=30mm,bmargin=40mm]{geometry}

\usepackage{latexsym,amsfonts,amsmath,amssymb,epsfig,tabularx,amsthm,dsfont,mathrsfs}

\usepackage{enumerate}

\usepackage{titling}
\usepackage{hyperref}
\hypersetup{colorlinks=true,
  linkcolor=black,
  citecolor=black,
  urlcolor=blue}

\theoremstyle{plain}

\newtheorem{theorem}{Theorem}[section]
\newtheorem{lemma}[theorem]{Lemma}

\newtheorem{corollary}[theorem]{Corollary}

\theoremstyle{definition}
\newtheorem{remark}[theorem]{Remark}

\renewcommand{\P}{{\mathbb P}}
\newcommand{\expect}{\operatorname{\mathbb{E}}}

\DeclareMathOperator{\Uniform}{\mathcal{U}}

\newcommand{\iid}{\stackrel{\textup{iid}}{\sim}}

\newcommand{\dint}{\,\mathup{d}}

\newcommand{\Finput}{\mathcal{F}}

\newcommand{\Torus}{{\mathbb T}}

\DeclareMathOperator{\dist}{dist}

\DeclareMathOperator{\median}{med}

\newcommand{\deter}{\textup{det}}

\newcommand{\Wspace}{\mathcal{W}}
\newcommand{\MC}{\textup{MC}}
\newcommand{\linMC}{\textup{linMC}}

\newcommand{\prob}{\textup{prob}}

\DeclareMathOperator{\Int}{INT}

\newcommand{\eps}{\varepsilon}

\newcommand{\embed}{\hookrightarrow}
\renewcommand{\rho}{\varrho}

\renewcommand{\vec}{\boldsymbol}
\newcommand{\R}{{\mathbb R}}

\newcommand{\N}{{\mathbb N}}
\newcommand{\Z}{{\mathbb Z}}

\DeclareMathAlphabet{\mathup}{OT1}{\familydefault}{m}{n}

\newcommand{\wt}{\widetilde}
\newcommand{\widebar}[1]{\mbox{\kern1.5pt\hbox{\vbox{\hrule height 0.6pt \kern0.35ex
        \hbox{\kern-0.15em \ensuremath{#1 }\kern0.0em}}}}\kern-0.1pt}

\newlength{\fixboxwidth}
\usepackage{color}
\definecolor{darkgreen}{rgb}{0,0.5,0}
\definecolor{orange}{rgb}{1,0.4,0}

\title{Linear Monte Carlo quadrature\\ with optimal confidence intervals}

\date{\today}

\author{Robert J. Kunsch\thanks{RWTH Aachen University,
    Chair of Mathematics of Information Processing,
    Pontdriesch~10,
    52062 Aachen, Email: kunsch@mathc.rwth-aachen.de;
    TU Chemnitz, Faculty of Mathematics, Reichenhainer Str. 41, 09126 Chemnitz
  }
}

\begin{document}

\maketitle

\begin{abstract}
  We study the numerical integration of functions
  from isotropic Sobolev spaces $W_p^s([0,1]^d)$
  using finitely many function evaluations within randomized algorithms,
  aiming for the smallest possible probabilistic error guarantee $\eps > 0$
  at confidence level \mbox{$1-\delta \in (0,1)$}.
  For spaces consisting of continuous functions,
  non-linear Monte Carlo methods with optimal confidence properties have already been known,
  in few cases even linear methods that succeed in that respect.
  In this paper we promote a new method called \emph{stratified control variates} (SCV)
  and by it show that already linear methods achieve optimal probabilistic error rates
  in the high smoothness regime
  without the need to adjust algorithmic parameters to the uncertainty~$\delta$.
  We also analyse a version of SCV in the low smoothness regime
  where $W_p^s([0,1]^d)$ may contain functions with singularities.
  Here, we observe a polynomial dependence of the error on~$\delta^{-1}$
  which cannot be avoided for linear methods.
  This is worse than what is known to be possible using non-linear algorithms
  where only a logarithmic dependence on $\delta^{-1}$ occurs
  if we tune in for a specific value of $\delta$.
\end{abstract}

\textbf{Keywords.\;} Monte Carlo integration;
Sobolev functions;
information-based complexity;
linear methods;
asymptotic error;
confidence intervals.

\section{Introduction} \label{sec: intro}

We want to compute the integral
\begin{equation}   \label{eq:INT} 
  \Int f = \int_{G} f(\vec{x}) \dint \vec{x}
\end{equation} 
of $f\colon G \to \R$ from
a normed linear space $\Wspace$ of integrable functions 
defined on a domain $G \subset \R^d$,
in this paper we mainly restrict to the unit cube $G = [0,1]^d$.
The integral shall be approximated
using randomized \emph{linear} quadrature rules,
that is, measurable mappings
$Q_n \colon \Omega \times \Wspace \to \R$ of the shape
\begin{equation} \label{eq:Q=linMC}
  Q_n^{\omega}(f)
    = \sum_{i=1}^{n} w_i^{\omega} f(\vec{x}_i^{\omega}) \,,
\end{equation}
where $\omega \in \Omega$ represents the randomness
provided by a probability space $(\Omega,\Sigma,\P)$,
the nodes $\vec{x}_i^{\omega}$ shall be random variables on
the domain $G$,
and $w_i^{\omega} \in \R$ are suitable weights.
We study the \emph{probabilistic error} of such a method
for a given \emph{confidence level} $\delta \in (0,1)$ (or \emph{uncertainty} $\delta$)
and an input $f \in \Wspace$,
namely
\begin{equation} \label{eq:e(Q_n,del,f)}
  e(Q_n,\delta,f)
    := \inf \bigl\{\eps > 0 \colon \P(|Q_n(f) - \Int f| > \eps ) \leq \delta 
            \bigr\} \,.
\end{equation}
The accuracy of $Q_n$ on a given set $\Finput \subset \Wspace$ is defined via the worst case,
\begin{equation} \label{eq:e(Q_n,delta)}
  e(Q_n,\delta,\Finput)
    := \sup_{f \in \Finput} e(Q_n,\delta,f)\,,
\end{equation}
where the input set $\Finput$ is typically the unit ball of the function space $\Wspace$,
that is, $\Finput = \{f \in \Wspace \colon \|f\|_{\Wspace} \leq 1\}$,
in which case we will simply write $\Wspace$ instead of $\Finput$
on the left-hand side of \eqref{eq:e(Q_n,delta)}.
Alternatively, $\Finput$ could also consist of functions bounded by $1$ with respect to a semi-norm in $\Wspace$,
and indeed, the error bounds we prove in this paper equally hold for such larger input sets.
The cost of an algorithm \eqref{eq:Q=linMC} is dominated by the amount $n \in \N_0$
of function evaluations required by the method, the so-called \emph{cardinality}.
Aiming for the best possible approximation, 
we thus take the infimum over all linear methods $Q_n$ with cardinality~$n$,
\begin{equation} \label{eq:e^prob}
  e^{\linMC}_{\prob}(n,\delta,\Finput) := \inf_{Q_n} e_{\prob}(Q_n,\delta,\Finput) \,.
\end{equation}
If we allow general non-linear methods, we write $e^{\MC}_{\prob}(n,\delta,\Finput)$;
restricting to deterministic methods without $\omega$-dependence
leads to a quantity $e^{\deter}(n,\Finput)$.

The error criterion~\eqref{eq:e(Q_n,del,f)} is common in statistics,
usually expressed in terms of confidence intervals:
If $\eps \geq e(Q_n,\delta,\Finput) \in (0,\infty)$,
then \mbox{$[Q_n(f) - \eps, Q_n(f) + \eps]$} is a confidence interval 
at confidence level $1-\delta$
for the quantity $\Int(f)$, provided \mbox{$f \in \Finput$}.
In \emph{information-based complexity} (IBC), though,
the standard notion of \emph{Monte Carlo error}
is the root mean squared error
$\sqrt{\expect|Q_n(f) - \Int f|^2}$,
or simply the expected error~$\expect|Q_n(f) - \Int f|$,
see for instance~\cite{No88,TWW88}.
The paper~\cite{KNR18}, however, discussed an integration problem with non-convex input set
where positive results were only possible for the probabilistic error criterion,
it thus represents the more general yet challenging error criterion.
The subsequent publication \cite{KR19} seems to be the first thorough study
on classical integration problems in terms of the probabilistic error.
There, mainly isotropic Sobolev spaces~$W_p^s(G)$
on domains~$G \subseteq \R^d$ were considered. These are defined by
\begin{equation*}
  W_p^s(G)
    :=\biggl\{f \in L_p(G) \,\bigg|\,
              \|f\|_{W_p^s(G)}
                := \biggl(\sum_{\substack{\vec{\alpha} \in \N_0^d\\
                                   |\vec{\alpha}|_1 \leq s}}
                     \|D^{\vec{\alpha}} f\|_{L_p(G)}^p\biggr)^{1/p}
                < \infty
       \biggr\} \,,  
\end{equation*}
with integrability parameter $1 \leq p \leq \infty$
(with the usual modification for~\mbox{$p=\infty$})
and integer smoothness~$s \in \N_0$,
imposing Lebesgue integrability 
for weak partial derivatives
$D^{\vec{\alpha}}f = \partial_{x_1}^{\alpha_1} \cdots \partial_{x_d}^{\alpha_d} f$
for multi-indices $\vec{\alpha}=(\alpha_1,\dots,\alpha_d)\in\N_0^d$ 
with total degree $|\vec{\alpha}|_1 = \alpha_1 + \ldots + \alpha_d$ at most $s$.
Provided we have
sufficient smoothness, $s>d/p$,
hence guaranteeing continuity of functions, $W_p^s([0,1]^d) \embed C([0,1]^d)$,
the precise probabilistic error rate was determined in~\cite[Thm~4]{KR19}, namely
\begin{equation} \label{eq:erate-intro}
  e^{\MC}_{\prob}\left(n,\delta,W_p^s([0,1]^d)\right)
    \asymp
      \begin{cases}
        n^{-s/d} \cdot \min\left\{1,\, \left(\frac{\log \delta^{-1}}{n}\right)^{1-1/p} \right\}
          &\text{if } 1 \leq p < 2 \,, \\
        n^{-s/d} \cdot \min\left\{1,\, \sqrt{\frac{\log \delta^{-1}}{n}}\right\}
          &\text{if } p \geq 2 \,.
      \end{cases}
\end{equation}
These rates reveal the benefit from Monte Carlo compared to deterministic quadrature
in the regime $n \succ \log \delta^{-1}$.
Upper bounds were achieved by a variant of \emph{control variates} (CV):
Half of the information budget is spent on an~$L_{p'}$-approximation~$g$
of the function~$f$, where $p' := \min\{2,p\}$ and the integral $\Int g$ is known;
the other half of the budget is used on estimating the integral
of the $L_q$-residual~$f-g$ via a \emph{non-linear} Monte Carlo method,
namely, the \emph{median-of-means} (MoM),
which is a probability amplification scheme
applied to the standard Monte Carlo method, see~\eqref{eq:CV+MoM} for details.
For smoothness $s = 1$ however,
\emph{stratified sampling} was shown to achieve optimal rates
without the need for non-linear algorithmic features.
Instead of employing a probability amplification scheme like the median in some stage of the algorithm,
for the analysis of stratified sampling Hoeff\-ding's inequality was used,
one of the best known concentration inequalities in probability theory.
This raises the general question in which cases linear Monte Carlo methods
have the potential of optimal confidence properties.

In the present article we give a positive result stating
that linear Monte Carlo methods achieve optimal probabilistic error rates
for Sobolev spaces $W_p^s([0,1]^d)$ of continuous functions,
see Theorem~\ref{thm:linMC-cont}.
This is achieved by a new method we call \emph{stratified control variates} (SCV)
which combines the two aforementioned classical variance reduction techniques,
see Section~\ref{sec:UBs} for a detailed description.
The new approach exploits that a control variate $g$ for $f$ may locally have
much better approximation properties than globally,
say, the residual $f-g$ may be bounded on sub-domains of $[0,1]^d$
but on the whole domain we can only give an \mbox{$L_2$-bound} with considerable approximation rates.
Besides linearity and unbiasedness, SCV has yet another main advantage
over the combination of control variates with the median-of-means (CV+MoM),
namely, SCV exhibits optimal confidence properties universally for all uncertainty levels $\delta$,
whereas the non-linear approach CV+MoM contains a parameter that needs to be adjusted to $\delta$:
In the MoM stage we take the median of $k \approx 2 \log_2(2\delta)^{-1}$ repetitions
of the standard Monte Carlo method applied to the residual $f-g$,
see \cite[Sec~3]{KR19} for details.

We also analyse SCV for the low smoothness regime
where $W_p^s([0,1]^d)$ is not embedded in the space of continuous functions.
Here, the probabilistic error guarantees are worse than what is known to be possible with non-linear methods,
namely, for linear algorithms we find a polynomial dependence on $\delta^{-1}$.
In Section~\ref{sec:LBs} we show that this worse tail behaviour cannot be avoided for linear methods,
yet, the precise asymptotics of the optimal joint $(n,\delta)$-dependence remains vague.

We close the paper with some numerical experiments comparing different integration methods 
that use the same control variate,
see Section~\ref{sec:numerics}.
The results underscore the superiority of our new method.

\textbf{Asymptotic notation:}
  For functions $e,f \colon \N\times (0,1) \to \R$
  we use the notation
  $e(n,\delta) \preceq f(n,\delta)$,
  meaning that there is some $n_0\in\N$ and $\delta_0\in(0,1)$ such that
  $e(n,\delta) \leq c f(n,\delta)$
  for all $n \geq n_0$
  and $\delta \in (0,\delta_0)$
  with some constant~$c > 0$ that may depend on other parameters.
  Asymptotic equivalence $e(n,\delta) \asymp f(n,\delta)$
  is a shorthand for \mbox{$e(n,\delta)\preceq f(n,\delta) \preceq e(n,\delta)$}.

\section{Stratified control variates}
\label{sec:UBs}

This section is devoted to upper bounds for the numerical integration
of functions from isotropic Sobolev spaces $W_p^s([0,1]^d)$,
where $1 \leq p \leq \infty$ and $d,s \in \N$.
This is done by analysing a new algorithm we call \emph{stratified control variates}
which combines the well known ideas of control variates
and stratified sampling.

For the analysis we consider the usual Sobolev semi-norms
\begin{equation} \label{eq:seminorm}
  |f|_{W_p^s(G)} := \biggl(\sum_{\substack{\vec{\alpha} \in \N_0^d\\
      |\vec{\alpha}|_1 = s}}
  \|D^{\vec{\alpha}} f\|_{L_p(G)}^p\biggr)^{1/p} \,.
\end{equation}
This semi-norm $|f|_{W_p^s(G)}$ is finite if and only if $f \in W_p^s(G)$,
its zero set is the space $\mathcal{P}^s(\R^d)$
of $d$-variate polynomials of total degree smaller than $s$,
that is, $|g|_{W_p^s(G)} = 0$ precisely for $g \in \mathcal{P}^s(G)$.
This polynomial space has the dimension
\begin{equation} \label{eq:n0}
  n_0 = n_0(s,d) := \dim \mathcal{P}^s(\R^d) = \binom{s+d-1}{d} \,.
\end{equation}
Piecewise interpolation with such polynomials will
provide the control variate. 
There exists a so-called \emph{$s$-regular set}
$\mathcal{X} = \mathcal{X}^s := \{\vec{x}_j\}_{j=1}^{n_0} \subset [0,1]^d$
such that for any given values $\left(y_j\right)_{j=1}^{n_0} \subset \R$
there is exactly one polynomial $g \in \mathcal{P}^s(\R^d)$ satisfying $g(\vec{x}_j) = y_j$ for all $j=1,\ldots,n_0$,
see for instance~\cite[Appendix~A]{Vyb07}.
When dealing with continuous functions, i.e.~$W_p^s([0,1]^d) \embed C([0,1]^d)$,
which holds for $s > d/p$, 
see~\cite[Thm~4.12~(1)]{AF03},
a key observation is that
\begin{equation} \label{eq:altnormC}
  \|f\|_{W_p^s([0,1]^d)}' := \sum_{j=1}^{n_0} |f(\vec{x}_j)| + |f|_{W_p^s([0,1]^d)}
\end{equation}
is an equivalent norm on $W_p^s([0,1]^d)$, see, for instance \cite[eq~(3.1.11)]{Cia78}.
We write $P_{\mathcal{X}}f$ for the polynomial $g \in \mathcal{P}^s(\R^d)$
that interpolates the function values of $f$ on $\mathcal{X}$,
i.e.\ for all $j=1,\ldots,n_0$ we match $g(\vec{x}_j) = f(\vec{x}_j)$.
For the difference $f-P_{\mathcal{X}}f$ we have 
\begin{equation} \label{eq:interpolWnorm}
  \|f-P_{\mathcal{X}}f\|_{W_p^s([0,1]^d)}' = |f|_{W_p^s([0,1]^d)} \,.
\end{equation}
Since we assumed $W_p^s([0,1]^d)$ to be continuously embedded into the space of continuous functions,
there exists a constant $c_{p,\infty}^{s,d} > 0$ such that
\begin{equation} \label{eq:interpolLooErr}
  \|f-P_{\mathcal{X}}f\|_{L_\infty([0,1]^d)} \leq c_{p,\infty}^{s,d} \, |f|_{W_p^s([0,1]^d)} \,.
\end{equation}
If, however, the function spaces contains discontinuous functions and we only have
the embedding
$W_p^s([0,1]^d) \embed L_q([0,1]^d)$ for some $q \in (p, \infty)$,
namely for $\frac{1}{q} \geq \frac{1}{p} - \frac{s}{d}$, see~\cite[Thm~4.12, eq~(5)]{AF03},
we need a randomly shifted point set
\begin{equation} \label{eq:X_xi}
  \mathcal{X}_{\vec{\xi}} := \left\{\frac{1}{2}\left(\vec{x}_j + \vec{\xi}\right)\right\}_{j=0}^{n_0} \subset [0,1]^d
  \qquad\text{with } \vec{\xi} \sim \Uniform([0,1]^d) \,.
\end{equation}
From Heinrich~\cite[eq~(21)]{He08SobI} we know that for such a randomized interpolation we have the bound
\begin{equation} \label{eq:interpolLqErr}
  \left(\expect\|f - P_{\mathcal{X}_{\vec{\xi}}} f\|_{L_q([0,1]^d)}^q\right)^{1/q}
    \leq c_{p,q}^{s,d}\, |f|_{W_p^s([0,1]^d)} \,,
\end{equation}
with a suitable constant $c_{p,q}^{s,d} > 0$.
In what follows we describe the algorithm with randomly shifted point set $\mathcal{X}_{\vec{\xi}}$,
keeping in mind that for spaces of continuous functions the deterministic point set $\mathcal{X}$ can do the job as well.

The semi-norm $|\cdot|_{W_p^s(G)}$ satisfies two important properties that are important
when decomposing the domain $[0,1]^d$ into $m^d$ essentially disjoint sub-cubes
\begin{equation} \label{eq:subcube}
G_{\vec{i}} = G_{m,\vec{i}} := \prod_{j=1}^d \left[\frac{i_j}{m} , \frac{i_j + 1}{m}\right]
\end{equation}
for $\vec{i} \in \{0,\ldots,m-1\}^d =: [0:m)^d$, namely, the decomposition property
\begin{equation} \label{eq:|decomp|}
|f|_{W_p^s([0,1]^d)}
= \left\|\left(|f|_{W_p^s(G_{\vec{i}})}\right)_{\vec{i} \in [0:m)^d}\right\|_{\ell_p}\,,
\end{equation}
and the scaling property
\begin{equation} \label{eq:scaling}
\left|f \circ \Phi_{m,\vec{i}}\right|_{W_p^s([0,1]^d)} = m^{d/p - s} |f|_{W_p^s(G_{\vec{i}})}
\quad\text{where}\quad
\Phi_{m,\vec{i}}(\vec{x}) := \frac{\vec{x}+\vec{i}}{m} \,.
\end{equation}
We define (randomized) interpolates of $f$ on the sub-cubes $G_{\vec{i}}$ via
\begin{equation} \label{eq:g_i}
  g_{\vec{i},\vec{\xi}} := \bigl(P_{\mathcal{X}_{\vec{\xi}}}(f \circ \Phi_{m,\vec{i}})\bigr) \circ \Phi_{m,\vec{i}}^{-1} \,.
\end{equation}
The integrals of the local interpolation polynomials,
\begin{equation} \label{eq:a_i}
  a_{\vec{i}}
    := \int_{[0,1]^d} P_{\mathcal{X}_{\vec{\xi}}}(f \circ \Phi_{m,\vec{i}}) \dint\vec{x}
    = m^d \int_{G_{\vec{i}}} g_{\vec{i},\vec{\xi}}(\vec{x}) \dint\vec{x} \,,
\end{equation}
can be computed exactly and provide initial approximations
for the mean value of~$f$ on the sub-cubes $G_{\vec{i}}$.
Finally, we estimate the residual via stratified sampling
by taking independent samples $\vec{X}_{\vec{i}}^{(j)} \sim \Uniform(G_{\vec{i}})$ that are also independent from $\vec{\xi}$,
\emph{stratified control variates} (SCV) is then defined as the linear and unbiased method
\begin{equation} \label{eq:SCV}
  A_{m,s}^{\textup{SCV}}(f)
    := \frac{1}{m^d} \sum_{\vec{i} \in [0:m)^d}
          \left(a_{\vec{i},\vec{\xi}} 
                  + \frac{1}{n_0(s,d)} \sum_{j=1}^{n_0(s,d)}
                      [f - g_{\vec{i},\vec{\xi}}](\vec{X}_{\vec{i}}^{(j)})
          \right) \,.
\end{equation}
This method requires $n = 2 n_0 m^d$ function evaluations of $f$.
The decision to take $n_0$ random samples on each sub-cube $G_{\vec{i}}$
follows the usual heuristic for control variates to evenly distribute the information budget
on the two stages of approximating $f$ and of estimating the residual.
Stratified control variates can be described as the strategy of averaging
the means of $f$ on subcubes $G_{\vec{i}}$, the so-called \emph{strata},
where the means on the individual strata $G_{\vec{i}}$ are approximated via
a basic control variates method with fixed cardinality.
Note that we use the same shift $\vec{\xi}$ on all subcubes.

The first theorem gives upper bounds for spaces of continuous functions.

\begin{theorem} \label{thm:linMC-cont}
  Let $s,d \in \N$ and $1 \leq p \leq \infty$ with $s > d/p$.
  Then there exists a family $(Q_n)_{n \in \N}$
  of linear randomized quadrature rules 
  that achieve optimal probabilistic error rates,
  namely, for $n \geq 2n_0(s,d)$ and all $\delta \in (0,\frac{1}{4})$ we have
  \begin{align*}
    e\left(Q_n,\delta,W_p^s([0,1]^d)\right)
      &\asymp e_{\prob}^{\linMC}\left(n,\delta,W_p^s([0,1]^d)\right)
      \asymp e_{\prob}^{\MC}\left(n,\delta,W_p^s([0,1]^d)\right) \\
      &\asymp
        \begin{cases}
          n^{-s/d} \cdot \min\left\{1,\, \left(\frac{\log \delta^{-1}}{n}\right)^{1-1/p} \right\}
            &\text{if } 1 \leq p < 2 \,, \\
          n^{-s/d} \cdot \min\left\{1,\, \sqrt{\frac{\log \delta^{-1}}{n}}\right\}
            &\text{if } p \geq 2 \,.
        \end{cases}    
  \end{align*}
\end{theorem}
\begin{proof}
  Lower bounds for arbitrary methods and $\delta \in (0,\frac{1}{4})$ can be found in \cite[Thm~1]{KR19}.
  For upper bounds we employ the SCV algorithm \eqref{eq:SCV}
  with $m := \bigl\lfloor (\frac{n}{2n_0})^{1/d} \bigr\rfloor$.
  
  The assumption $s/d > 1/p$ implies that $W_p^s([0,1]^d) \embed C([0,1]^d)$,
  so we may confine ourselves to interpolation with a fixed point set $\mathcal{X}$
  and use a scaled version of~\eqref{eq:interpolLooErr}
  with deterministic interpolants $g_{\vec{i}}$ in place of~\eqref{eq:g_i}.
  In detail, using the scaling property~\eqref{eq:scaling}, we get
  \begin{align}
    b_{\vec{i}}
      &:= \|f - g_{\vec{i}}\|_{L_{\infty}(G_{\vec{i}})} \nonumber\\
      &= \|f \circ \Phi_{m,\vec{i}} - P_{\mathcal{X}}(f \circ \Phi_{m,\vec{i}})\|_{L_\infty([0,1]^d)} \nonumber\\
      &\leq c_{p,\infty}^{s,d} \, |f \circ \Phi_{m,\vec{i}}|_{W_p^s([0,1]^d)} \nonumber\\
      &= c_{p,\infty}^{s,d} \, m^{-(s-d/p)} \, |f|_{W_p^s(G_{\vec{i}})} \,. \label{eq:b_i}
  \end{align}
  First, these bounds give a worst case bound on the error of SCV, namely, in combination with \eqref{eq:|decomp|} we estimate
  \begin{align}
    \left|A_{m,s}^{\textup{SCV}}(f) - \Int f \right|
      &\leq 2m^{-d}\sum_{\vec{i} \in [0:m)^d} b_i \nonumber\\
      &\leq 2m^{-d/p}\left(\sum_{\vec{i} \in [0:m)^d} b_i^p\right)^{1/p} \nonumber\\
      &\leq 2 c_{p,\infty}^{s,d} \, m^{- s} \, |f|_{W_p^s([0,1]^d)} \,. \label{eq:wor}
  \end{align}
  
  The bounds \eqref{eq:b_i} on the range of the residual estimators also play a role in Hoeffding's inequality
  by which we find the probabilistic bound
  \begin{equation*}
    \P\left\{\left|A_{m,s}^{\textup{SCV}}(f) - \Int f \right| > \eps \right\}
        \leq 2 \exp\left(-\frac{n_0^2 m^{2d} \eps^2}{2 \sum\limits_{\vec{i} \in [0:m)^d} b_{\vec{i}}^2}\right) 
        \stackrel{!}{=} \delta \,.
  \end{equation*}
  Resolving for $\eps$, we obtain the error bound
  \begin{equation*}
    \eps = \frac{1}{n_0 m^d} \cdot \sqrt{2 \log \frac{2}{\delta}} \cdot \sqrt{\sum_{\vec{i} \in [0:m)} b_{\vec{i}}^2} \,.
  \end{equation*}
  This is useful for $p \geq 2$, where with \eqref{eq:|decomp|} and \eqref{eq:b_i} we estimate
  \begin{equation*}
    \left(\sum_{\vec{i} \in [0:m)^d} b_{\vec{i}}^2\right)^{1/2}
      \leq m^{d(1/2-1/p)} \left(\sum_{\vec{i} \in [0:m)^d} b_{\vec{i}}^p\right)^{1/p}
      \leq c_{p,\infty}^{s,d} \, m^{-(s-d/2)} \, |f|_{W_p^s([0,1]^d)} \,,
  \end{equation*}
  For $1 \leq p < 2$ we use an alternative Hoeffding type inequality, see Lemma~\ref{lem:pnorm-Hoeffding}, providing the error bound
  \begin{equation*}
    \eps 
      = \frac{3}{n_0 m^d} \cdot \left(2 \log \frac{2}{\delta}\right)^{1-1/p} \cdot \left(\sum_{\vec{i} \in [0:m)} b_{\vec{i}}^p\right)^{1/p} \,,
  \end{equation*}
  where \eqref{eq:|decomp|} and \eqref{eq:b_i}, again,
  imply an appropriate relation to the Sobolev norm,
  \begin{equation*}
    \left(\sum_{\vec{i} \in [0:m)^d} b_{\vec{i}}^p\right)^{1/p}
      \leq c_{p,\infty}^{s,d} \, m^{-(s-d/p)} \, |f|_{W_p^s([0,1]^d)} \,.
  \end{equation*}
  Writing $p' := \min\{2,p\}$ to accommodate both regimes,
  for $f$ from the unit ball of $W_p^s([0,1]^d)$ we thus get
  the following bound holding with probability $1-\delta \in (0,1)$:
  \begin{equation*}
    \left|A_{m,s}^{\textup{SCV}}(f) - \Int f \right|
      \leq \frac{c_{p,\infty}^{s,d}}{n_0} \cdot m^{-s-d(1-1/p')} 
        \cdot \left(2 \log \frac{2}{\delta}\right)^{1-1/p'} \,.
  \end{equation*}
  With $n \asymp m^d$ we achieve the randomized rates as claimed,
  where \eqref{eq:wor} is the bound we fall back to for very small~$\delta$.
\end{proof}

\begin{remark}[Stratified sampling]
  For spaces $W_p^s([0,1]^d)$ of smoothness $s=1$ and integrability $d > p$,
  stratified sampling without control variates
  already yields optimal probabilistic error rates.
  This was shown in \cite[Thm~6]{KR19}
  for integrability $p \geq 2$,
  in the case of low integrability $1 < p < 2$, though,
  namely spaces $W_p^1([0,1])$ of univariate functions,
  a gap in the power of $\log \delta^{-1}$ remained open.
  With the help of the new $p$-norm version of Hoeffding's inequality, see Lemma~\ref{lem:pnorm-Hoeffding},
  this gap can be closed with the lines following the proof of Theorem~\ref{thm:linMC-cont},
  implying that stratified sampling is optimal in that case as well.
\end{remark}

The second theorem gives probabilistic integration rates for spaces of functions that are not necessarily continuous.

\begin{theorem} \label{thm:s<d/p}
  Let $s,d \in \N$ and $1 \leq p < \infty$ with $s < d/p$.
  Then there exists a family $(Q_n)_{n \in \N}$
  of linear randomized quadrature rules 
  that for $n \geq 2n_0(s,d)$ and all $\delta \in (0,1)$
  achieve the following probabilistic error rates:
  \begin{align*}
    e_{\prob}^{\linMC}\left(n,\delta,W_p^s([0,1]^d)\right)
      &\leq e\left(Q_n,\delta,W_p^s([0,1]^d)\right) \\
      &\preceq \begin{cases}
                n^{-\left(\frac{s}{d} + 1 - \frac{1}{p}\right)} \cdot \delta^{-\left(\frac{1}{p} - \frac{s}{d}\right)}
                  &\text{if } 1 \leq p \leq 2 \,, \\
                n^{-\left(\frac{s}{d} + \frac{1}{2}\right)} \cdot \delta^{-\left(\frac{1}{p} - \frac{s}{d}\right)}
                  &\text{if } 2 \leq p < \infty \,.
              \end{cases}
  \end{align*}
\end{theorem}
\begin{proof}
  We employ the SCV algorithm~\eqref{eq:SCV} with $m := \bigl\lfloor (\frac{n}{2n_0})^{1/d} \bigr\rfloor$
  and randomized interpolation.
  In what follows, we write $\expect^{\vec{\xi}}$ for the expectation that averages over the random shift $\vec{\xi}$,
  while $\expect[\,\cdot\,| \vec{\xi}]$ is the conditional expectation for fixed $\vec{\xi}$,
  hence, producing a random variable derived from $\vec{\xi}$.
  
  For $q > p$ with $\frac{1}{q} = \frac{1}{p} - \frac{s}{d} > 0$
  the embedding $W_p^s([0,1]^d) \embed L_q([0,1]^d)$ holds
  and we may apply~\eqref{eq:interpolLqErr}.
  Using randomized interpolants $g_{\vec{i},\vec{\xi}}$ on sub-domains $G_{\vec{i}}$, see~\eqref{eq:g_i},
  the scaling property~\eqref{eq:scaling} leads to
  \begin{align}
    \left(\expect^{\vec{\xi}} \|f - g_{\vec{i},\vec{\xi}}\|_{L_q(G_{\vec{i}})}^q\right)^{1/q}
      &= m^{-d/q} \left(\expect^{\vec{\xi}} \|f\circ \Phi_{m,\vec{i}} - P_{\mathcal{X}_{\vec{\xi}}}(f\circ \Phi_{m,\vec{i}})
                                    \|_{L_q(G_{\vec{i}})}^q
                     \right)^{1/q} \nonumber\\
      &\leq c_{p,q}^{s,d} \, m^{-d/q} \, |f \circ \Phi_{m,\vec{i}}|_{W_p^s([0,1]^d)} \nonumber\\
      &= c_{p,q}^{s,d} \, m^{d\left(\frac{1}{p} - \frac{1}{q}\right) - s} \, |f|_{W_p^s(G_{\vec{i}})} \,. \label{eq:c_i}
  \end{align}
  
  In order to estimate the $L_q$-error of the SCV algorithm,
  we need to understand the $q$-th moment of the random variables
  $Z_{\vec{i},\vec{\xi}}^{(j)} := [f - g_{\vec{i},\vec{\xi}}](\vec{X}_{\vec{i}}^{(j)})$
  with 
  \begin{equation*}
    \left(\expect\left[\bigl|Z_{\vec{i},\vec{\xi}}^{(j)}\bigr|^q \middle| \vec{\xi}\right]\right)^{1/q}
      = m^{d/q} \, \|f - g_{\vec{i},\vec{\xi}}\|_{L_q(G_{\vec{i}})} \,.
  \end{equation*}
  For any fixed random shift $\vec{\xi}$ these random variables are independent
  and the algorithm is unbiased, hence, we can
  apply Lemma~\ref{lem:MZ-ineq}.
  Writing $q' := \min\{2,q\}$ we obtain
  \begin{align}
    &\left(\expect\Bigl[\bigl|A_{m,s}^{\textup{SCV}}(f) - \Int f\bigr|^q \Big| \vec{\xi}\right]\Bigr)^{1/q} \nonumber\\
      &\leq \frac{c_q}{n_0 m^d}
              \left(\sum_{\vec{i}\in[0:m)^d} \sum_{j=1}^{n_0}
                      \left(\expect\left[\bigl|Z_{\vec{i},\vec{\xi}}^{(j)}\bigr|^q \middle| \vec{\xi}\right]\right)^{q'/q}
              \right)^{1/{q'}} \nonumber\\
      &= c_q
          \cdot n_0^{-(1-1/q')} \cdot m^{-d(1-1/q)}
          \, \left(\sum_{\vec{i}\in[0:m)^d} \|f - g_{\vec{i},\vec{\xi}}\|_{L_q(G_{\vec{i}})}^{q'}
              \right)^{1/q'} \,.
        \label{eq:q-moment|xi}
  \end{align}
  We also need the $q$-th moment of~\eqref{eq:q-moment|xi}
  with respect to the expectation over the random shift $\vec{\xi}$
  in order to establish relations to the local semi-norms via~\eqref{eq:c_i}.
  For $q' = q \leq 2$ this causes no problems,
  but for $q' = 2 < q$ it requires the use of Lemma~\ref{lem:RV-triangle-ineq}
  with random variables $Y_{\vec{i}} := \|f - g_{\vec{i},\vec{\xi}}\|_{L_q(G_i)}$,
  thus
  \begin{multline}
  \left(\expect^{\vec{\xi}} \left(\sum_{\vec{i}\in[0:m)^d}
    \|f - g_{\vec{i},\vec{\xi}}\|_{L_q(G_{\vec{i}})}^{q'}\right)^{q/q'}
  \right)^{1/q}
    \leq \left(\sum_{\vec{i}\in[0:m)^d}
          \left(\expect^{\vec{\xi}} \|f - g_{\vec{i},\vec{\xi}}\|_{L_q(G_{\vec{i}})}^q
          \right)^{q'/q}
    \right)^{1/q'} \\
    \leq c_{p,q}^{s,d} \, m^{d\left(\frac{1}{p} - \frac{1}{q}\right) - s}
              \left(\sum_{\vec{i}\in[0:m)^d} |f|_{W_p^s(G_{\vec{i}})}^{q'}
      \right)^{1/q'} \,.
      \label{eq:ell_q'-sum}
  \end{multline}
  Consider $\vec{x} \in \R^{M}$.
  If $p < 2$, we have $p < q'$, hence, $\|\vec{x}\|_{\ell_{q'}} \leq \|\vec{x}\|_{\ell_p}$.
  If $p \geq 2$, we have $q' = 2$, hence, 
  $\|\vec{x}\|_{\ell_{q'}} = \|\vec{x}\|_{\ell_2} \leq M^{\left(\frac{1}{2} - \frac{1}{p}\right)} \|\vec{x}\|_{\ell_p}$.
  Writing $p' := \min\{2,p\}$, with 
  $\vec{x} = \bigl(|f|_{W_p^s(G_{\vec{i}})}\bigr)_{\vec{i} \in [0:m)^d}$
  and $M = m^d$, and with the help of the decomposition property~\eqref{eq:|decomp|},
  we estimate
  \begin{align}
    \left(\sum_{\vec{i}\in[0:m)^d} |f|_{W_p^s(G_{\vec{i}})}^{q'}
    \right)^{1/q'}
      &\leq m^{d\left(\frac{1}{p'} - \frac{1}{p}\right)}
              \left(\sum_{\vec{i}\in[0:m)^d} |f|_{W_p^s(G_{\vec{i}})}^{p}
              \right)^{1/p} \nonumber\\
      &= m^{d\left(\frac{1}{p'} - \frac{1}{p}\right)} \,
            |f|_{W_p^s([0,1]^d)} \,. \label{eq:ell_q'_vs_seminorm}
  \end{align}
  For inputs from the unit ball, that is, $|f|_{W_p^s([0,1]^d)} \leq \|f\|_{W_p^s([0,1]^d)} \leq 1$,
  combining \eqref{eq:q-moment|xi}, \eqref{eq:ell_q'-sum}, and \eqref{eq:ell_q'_vs_seminorm},
  we then have
  \begin{align*}
    \left(\expect\bigl|A_{m,s}^{\textup{SCV}}(f) - \Int f\bigr|^q\right)^{1/q}
      &\leq \underbrace{c_q\,c_{p,q}^{s,d}\cdot n_0^{- (1 - 1/q')}}_{=: C_{p,q}^{s,d}}
            \cdot \, m^{-s - d(1 - 1/p')} \,.
  \end{align*}
  Via Markov's inequality, for an error threshold $\eps > 0$
  we get the following bound on the failure probability:
  \begin{equation*}
    \P\left\{\bigl|A_{m,s}^{\textup{SCV}}(f) - \Int f\bigr| > \eps\right\}
      \leq \left(C_{p,q}^{s,d} \cdot m^{-s - d(1-1/p')} \cdot \eps^{-1}\right)^q \,.
  \end{equation*}
  This is guaranteed to be no bigger than a given $\delta \in (0,1)$ for
  \begin{equation*}
    \eps = C_{p,q}^{s,d} \cdot m^{-s-d(1-1/p')} \cdot \delta^{-1/q} \,.
  \end{equation*}
  With $\frac{1}{q} = \frac{1}{p} - \frac{s}{d}$ and $n \asymp m^d$, we obtain the desired bounds.
\end{proof}

\begin{remark}[Comparison with non-linear methods]
  The above theorem gives error bounds of the shape
  \begin{equation*}
    e_{\prob}^{\linMC}(n,\delta) \preceq n^{-\rho} \cdot \delta^{-r} \,,
  \end{equation*}
  where the main rate is a number $\rho \in (0,1)$,
  and the tail of the error distribution becomes thinner the larger the regularity,
  namely, $r \searrow 0$ for, say, $p \nearrow d/s$ with fixed $d$ and $s$.
  A small value of $r$ close to $0$ is desirable
  as this means that the size of the confidence intervals grows at a slower rate
  when we decrease the acceptable uncertainty~$\delta$.
  
  Recall that applying, for instance, the median trick
  to a given family $(Q_n)_{n \in \N}$ of linear methods 
  can further reduce the dependence on $\delta^{-1}$,
  in our situation we would find
  \begin{equation*}
    e_{\prob}^{\MC}(n,\delta) \preceq \left(\frac{\log \delta^{-1}}{n}\right)^{\rho} \,,
  \end{equation*}
  see~\cite[Thm~2]{KR19}.
  For this we need $k \asymp \log \delta^{-1}$ independent repetitions of the linear method $Q_{n_1}$ of which we take the median,
  the resulting method is non-linear and uses $n = k n_1$ samples.
  The optimal $\delta$-dependence for non-linear methods in the low smoothness regime, though, remains an open problem.
\end{remark}

\begin{remark}
  The SCV algorithm is exact on the space $\mathcal{P}^s(\R^d)$ of polynomials of total degree less than $s$.
  Furthermore, the bounds in Theorems~\ref{thm:linMC-cont} and~\ref{thm:s<d/p} still hold
  if, instead of the classical unit ball in $W_p^s([0,1]^d)$, we consider the larger input set
  \begin{equation*}
    \Finput' := \left\{f \in W_p^s([0,1]^d) \colon |f|_{W_p^s([0,1]^d)} \leq 1 \right\} \,.
  \end{equation*}
  Restricting to functions from the unit ball, however,
  we would have
  \begin{equation*}
    |\Int f| \leq \|f\|_{L_p([0,1]^d)} \leq \|f\|_{W_p^s([0,1]^d)} \leq 1 \,,
  \end{equation*}
  hence, the zero algorithm $A_0(f) = 0$ would have an error of at most $1$,
  the so-called \emph{initial error}.
  An error bound of the shape $n^{-\rho} \cdot \delta^{-r}$ as it is given in Theorem~\ref{thm:s<d/p}
  would exceed the initial error for very small uncertainty levels $\delta \prec n^{-\rho/r}$.
  Hence, instead of performing SCV one could simply return $0$.
  We might therefore give an upper bound
  \begin{equation*}
    e_{\prob}^{\linMC}(n,\delta) \preceq \min\left\{1,\, n^{-\rho} \cdot \delta^{-r}\right\} \,.
  \end{equation*}
  The zero algorithm is linear but it introduces a bias, further, we need to decide,
  depending on $\delta$,
  whether we use SCV or just return $0$.
  The rates given by the Theorem, however, hold for SCV without the need to adjust anything to $\delta$.
  Besides, the trick of returning $0$ in order to reduce the error
  whenever a small uncertainty $\delta$ is given,
  does not work if we consider the larger input set $\Finput'$ with the semi-norm bound.
  Furthermore, SCV is a method with good confidence properties for any $\delta$
  without the need to adjust the algorithm
  (the only decision is to be made on the smoothness $s$ we aim to exploit,
  which is determined by the degree of interpolation).
  As Theorem~\ref{thm:delta-LB} will show,
  fixing a method for given cardinality~$n$,
  the $\delta$-dependence we obtained for SCV is of optimal order
  among all linear methods.
  Finally, at least for low integrability~$1 \leq p \leq 2$,
  we find matching lower bounds for the
  joint $(n,\delta)$-dependence of the
  error of SCV, see Theorem~\ref{thm:SCV-LB}.
\end{remark}

For the sake of completeness we provide a result for boundary smoothness.

\begin{corollary} \label{cor:s=d/p}
  Let $s,d \in \N$ and $1 < p < \infty$ with $s = d/p$.
  Then there exists a family $(Q_n)_{n \in \N}$
  of linear randomized quadrature rules
  such that for $n \geq 2n_0(s,d)$ 
  we achieve probabilistic error rates
  with sub-polynomial dependence on $\delta^{-1}$,
  namely, for all $r > 0$ and $\delta \in (0,1)$ we have
  \begin{align*}
    e_{\prob}^{\linMC}\left(n,\delta,W_p^s([0,1]^d)\right)
      &\leq e\left(Q_n,\delta,W_p^s([0,1]^d)\right) \\
      &\preceq \begin{cases}
                n^{-1} \cdot \delta^{-r}
                  &\text{if } 1 < p \leq 2 \,, \\
                n^{-\left(\frac{1}{p} + \frac{1}{2}\right)} \cdot \delta^{-r}
                  &\text{if } 2 \leq p < \infty \,,
              \end{cases}
  \end{align*}
  where the implicit constant depends on $r$, as well as on the space parameters $p$, $s$, and $d$.
\end{corollary}
\begin{proof}
  If $s = d/p$, then $W_p^s([0,1]^d) \embed L_q([0,1]^d)$ for all $q < \infty$.
  We follow the lines in the proof of Theorem~\ref{thm:s<d/p} with $r = \frac{1}{q}$ for $q \in [2,\infty)$
  and achieve the claimed bounds for $r \in (0,\frac{1}{2}]$.
  The bounds for larger $r$ are trivial.
\end{proof}

\begin{remark}[Integrability $p=1$]
  The case of integrability $p=1$ and $s=d$ was not included in Corollary~\ref{cor:s=d/p}
  because there deterministic quadrature rules already achieve a rate
  that cannot be improved by randomization:
  \begin{equation*}
    e_{\prob}^{\linMC}\bigl(n,\delta,W_1^s([0,1]^s)\bigr) \asymp e^{\deter}\bigl(n,W_1^s([0,1]^s)\bigr) \asymp n^{-1} \,.
  \end{equation*}
  For this statement we can make use of the embedding $W_1^s([0,1]^s) \embed C([0,1]^s)$,
  which is a special case mentioned in~\cite[Thm~4.12, eq~(1)]{AF03},
  and the proof of Theorem~\ref{thm:linMC-cont} contains all the subsequent arguments for the worst case guarantee.
  
  The case of $p=1$ and $s < d$, however, is contained in Theorem~\ref{thm:s<d/p}
  because for spaces of discontinuous functions
  deterministic quadrature cannot provide any worst case guarantee whatsoever
  while probabilistic guarantees are still possible in the randomized setting.
\end{remark}

\begin{remark}[Randomized interpolation for high smoothness]
  We discussed two versions of stratified control variates:
  One with deterministic interpolation for spaces of continuous functions,
  one with randomly shifted interpolation nodes in the low smoothness regime.
  In applications one might not know the precise class an integrand belongs to,
  so it is desirable to have a method that equally works in all settings.
  In fact, the latter version of SCV with random interpolation
  could also be applied in the high smoothness regime without any loss in the order of convergence.
  Indeed, for spaces $W_p^s([0,1]^d) \embed C([0,1]^d)$,
  the minimal constant in~\eqref{eq:interpolLooErr} depends continuously
  on the shift~$\vec{\xi}$ when taking a point set $\mathcal{X}_{\vec{\xi}}$ instead of $\mathcal{X}$.
  This is because for high smoothness $s > d/p$
  we can even find a H\"older exponent $\beta \in (0,1)$ such that $W_p^s([0,1]^d) \embed C^\beta([0,1]^d)$,
  namely $\beta \leq s - d/p$, see \cite[Thm~4.12, eq~(7)]{AF03}.
  Since the random shift $\vec{\xi}$ is from a compact domain,
  there exists a universal constant such that \eqref{eq:interpolLooErr} holds
  for all point sets $\mathcal{X}_{\vec{\xi}}$ with this constant.
\end{remark}

\begin{remark}[Other function spaces]
  The approach of using polynomial interpolation to reduce the function space norm to a semi norm
  with nice scaling and decomposition properties also works for more general function spaces,
  see for instance \cite[Appendix~A]{Vyb07} concerning Bezov spaces.
  The semi-norm representation of Bezov spaces in~\cite{Vyb07} suggests
  that even if the local polynomial interpolation we use is of degree $s$ or higher,
  we still obtain the same order for $W_p^s([0,1]^d)$, potentially with worse constants though.
  Spaces of dominating mixed smoothness pose yet another challenge
  and it would be interesting to see what can be achieved with a method in the spirit of SCV,
  see~\cite[Sec~4]{KR19} for an introductory discussion of such spaces
  in the context of the probabilistic error criterion.
\end{remark}

\begin{remark}[Other domains]
  The idea of stratified control variates (SCV) is not restricted to rectangular domains.
  If we have a triangulation of a domain~$G$,
  provided a bound on the maximal side-length of each simplex~$G_i \subset G$,
  we may find a scaling property similar to~\eqref{eq:scaling}.
  The final method will also need to take into account small variations
  in the volume of different simplices.
\end{remark}

\section{Lower bounds for low smoothness}
\label{sec:LBs}

So far we only discussed algorithms $Q_n$ with fixed cardinality $n \in \N$.
One might also consider algorithms with varying cardinality $\wt{n}(\omega)$
where $\expect \wt{n}(\omega) \leq n$ is the constraint we impose,
and this additional freedom is indeed present in many randomized integration methods,
see for instance~\cite{KrN17,mU17,KKNU18,KNW23,NuWil23}.
As discussed in \cite[Sec~2.1]{KR19}, however,
restricting to fixed cardinality is no major constraint.
Fixed cardinality facilitates the discussion of lower bounds.

Lower bounds for integration methods in Sobolev spaces are usually found
with the help of so-called bump functions,
for Sobolev spaces $W_p^s$ of smoothness $s \in \N_0$ one may define
a basic bump $\psi_0 : \R^d \to \R$ as
\begin{equation*}
  \psi_0(\vec{x})
    := \begin{cases}
          \left(1 - \|\vec{x}\|_{\ell_2}^2\right)^s
            &\text{for } \|\vec{x}\|_{\ell_2} := \sqrt{x_1^2+\ldots+x_d^2} \leq 1,\\
          0 &\text{else.}
        \end{cases}
\end{equation*}
For this function we have
\begin{equation*}
  1 \geq \psi_0(\vec{x}) \geq \left(\frac{3}{4}\right)^s =: b_0 > 0
  \qquad\text{for $\|\vec{x}\|_{\ell_2} \leq \frac{1}{2}$\,,}
\end{equation*}
and a positive integral (computable in polar coordinates using the Beta function),
\begin{equation*}
  \gamma_0 := \int_{\R^d} \psi_0(\vec{x}) \dint\vec{x}
    = \frac{\Gamma(s+1) \cdot \pi^{d/2}}{\Gamma(\frac{d}{2} + s + 1)}
    > 0 \,.
\end{equation*}
Further, partial derivatives $D^{\vec{\alpha}}\psi_0$ are continuous (in particular at the boundary)
up to the order $|\vec{\alpha}|_1 \leq s-1$, and still bounded and existing in a weak sense for $|\vec{\alpha}|_1 = s$,
hence, $\psi_0 \in W_p^s(\R^d)$ with a finite norm
\begin{equation*}
  c_0 = c_0(p,s,d) := \|\psi_0\|_{W_p^s([0,1]^d)} \in (0,\infty) \,.
\end{equation*}
Taking a scaling parameter $\sigma > 0$ and a shift $\vec{\xi} \in \R^d$
we define
\begin{equation} \label{eq:bumpscaling}
  \psi_{\sigma,\vec{\xi}}(\vec{x})
    := \frac{1}{c_0} \, \sigma^{- (d/p - s)} \,
          \psi_0\left(\frac{\vec{x} - \vec{\xi}}{\sigma}\right) \,,
\end{equation}
and thanks to proper normalisation we have $\|\psi_{\sigma,\vec{\xi}}\|_{W_p^s(\R^d)} \leq 1$,
compare the scaling property for semi norms~\eqref{eq:scaling}.
The fact that in the low smoothness regime $s < d/p$ we have
\begin{equation*} 
  \psi_{\sigma,\vec{\xi}}(\vec{\xi}) = \frac{1}{c_0} \cdot \sigma^{-(d/p-s)}
    \xrightarrow[\sigma \to 0]{} \infty \,,
\end{equation*}
reflects the potential unboundedness of functions from $W_p^s(\R^d)$.
This stands in contrast to the shrinking integral of such a bump,
\begin{equation} \label{eq:INT(bump)}
  \int_{\R^d} \psi_{\sigma,\vec{\xi}}(\vec{x}) \dint\vec{x}
    = \frac{\gamma_0}{c_0} \cdot \sigma^{s + d\left(1 - \frac{1}{p}\right)}
    \xrightarrow[\sigma \to 0]{} 0 \,.
\end{equation}

We start with a negative result on the $\delta$-dependence of the probabilistic error
for any fixed linear method $Q_n$
which inevitably will be polynomial in $\delta^{-1}$ in the case of low smoothness.

\begin{theorem} \label{thm:delta-LB}
  Let $s,d \in \N$ and $1\leq p <\infty$ with $s < d/p$.
  Then, for any linear randomized quadrature rule~$Q_n \neq 0$ that uses $n$ function values,
  there exist constants $c_n > 0$ and $\delta_0 \in (0,1)$ such that
  for all $\delta \in (0,\delta_0]$ we have
  \begin{equation*}
    e\left(Q_n,\delta,W_p^s([0,1]^d)\right)
      \geq c_n \cdot \delta^{-\left(\frac{1}{p} - \frac{s}{d}\right)} \,.
  \end{equation*}
\end{theorem}
\begin{proof}
  Let the method~$Q_n$ be given as in~\eqref{eq:Q=linMC}.
  For $t > 0$ define the random index set
  \begin{equation*}
    I^{\omega}_t := \left\{i \in \{1,\ldots,n\} \colon |w_i^{\omega}| \geq t \right\} \,.
  \end{equation*}
  Since $Q_n \neq 0$ there exists a threshold $t_0 > 0$ such that $\expect [\#I_{t_0}] > 0$.
  Fix this parameter $t_0$ and for $r > 0$ define the random set
  \begin{equation*}
    J^{\omega}_{r} := \Bigl\{i \in I^{\omega}_{t_0} \colon 
                            \sup_{j \in \{1,\ldots,m\} \setminus \{i\}} \wt{\dist}(\vec{x}_j^{\omega},\vec{x}_i^{\omega}) \geq r
                      \Bigr\} \,,
  \end{equation*}
  where $\wt{\dist}(\vec{x},\vec{y}) := \min_{\vec{k} \in \Z^d} \|\vec{x} - \vec{y} + \vec{k}\|_{\ell_2}$
  is the so-called ``wrap-around distance'' on the $d$-dimensional torus $\Torus^d$,
  that is the periodization of $[0,1]^d$ with opposing faces glued together, identifying the coordinates $0$ and $1$.
  We may assume that for $i \neq j$ the two integration nodes $\vec{x}_i^{\omega}$ and $\vec{x}_j^{\omega}$
  are almost surely distinct.
  (If not, $\vec{x}_j^{\omega}$ could be moved elsewhere and assigned
  a new weight $\wt{w}_j^\omega = 0$
  while the new weight for $\vec{x}_i^{\omega}$ is $\wt{w}_i^{\omega} := w_i^{\omega} + w_j^{\omega}$.)
  Hence, there exists an $r_0 > 0$ such that $\expect [\#J_{r_0}] > 0$.
  
  For $0 < \rho  \leq \min\{\frac{r_0}{3},\frac{1}{4}\}$
  let $\Xi_{\rho} = (\vec{\xi}_j)_{j=1}^{M(\rho)}$ be a $\rho$-cover of $\Torus^d$,
  that is, the union of the $\rho$-balls
  \begin{equation*}
    B_{\rho}(\vec{\xi}_j) := \left\{\vec{x} \in \Torus^d \colon \wt{\dist}(\vec{x},\vec{\xi}_j) < \rho\right\}
  \end{equation*}
  covers the entirety of $\Torus^d = [0,1]^d$.
  It is well known that there exists a $\rho$-cover such that the $\rho/2$-balls $B_{\rho/2}(\vec{\xi}_j)$
  do not overlap.
  Volume estimates show that such a $\rho$-covering has the size
  \begin{equation*}
    M(\rho) \leq V_d \cdot \left(\frac{2}{\rho}\right)^d\,,
  \end{equation*}
  where $V_d$ is the volume of the $d$-dimensional Euclidean unit ball.
  (On the $d$-torus $\Torus^d$ the volume of balls up to radius $\frac{1}{2}$ is the same as in $\R^d$.)
  By definition of the (random) index set $J_{r_0}$, and since $r_0 > 2\rho$,
  for each $\vec{\xi}_j$ we know that the random variable
  \begin{equation*}
    Y_j^\omega(\rho) := \#\left(B_{\rho}(\vec{\xi}_j) \cap \{\vec{x}_i^\omega \mid i \in J_{r_0}^{\omega}\}\right) 
  \end{equation*}
  takes only values $0$ or $1$.
  Adding things up, we find
  \begin{equation*}
    \expect [\#J_{r_0}] \leq \sum_{j=1}^{M(\rho)} \expect Y_j(\rho)\,,
  \end{equation*}
  hence there exists $j_0 \in \{1,\ldots,M(\rho)\}$ such that
  \begin{equation*}
    q(\rho) := \P\bigl\{Y_{j_0}(\rho) = 1\bigr\}
      = \expect Y_{j_0}(\rho) 
      \geq \frac{\expect [\#J_{r_0}]}{M(\rho)}
      > q_0 \cdot \rho^d
  \end{equation*}
  with a suitable constant $q_0 > 0$.
  If $Y_{j_0}^\omega(\rho) = 1$, let $i_0^\omega$ denote the unique index
  such that $\vec{x}_{i_0} = \vec{x}_{i_0^\omega}^{\omega} \in B_\rho(\vec{\xi}_{j_0})$,
  in which case for all $j \neq i_0$ we have $\vec{x}_j^\omega \notin B_{2\rho}(\vec{\xi}_{j_0})$
  since
  \begin{equation*}
    \wt{\dist}(\vec{x}_j,\vec{\xi}_{j_0})
      \geq \wt{\dist}(\vec{x}_j,\vec{x}_{i_0}) - \wt{\dist}(\vec{x}_{i_0},\vec{\xi}_{j_0}) 
      \geq r_0 - \rho
      \geq 2\rho \,.
  \end{equation*}
  Consider the periodized bump centred around $\vec{\xi}_{j_0}$ with scaling $\sigma = 2\rho$,
  \begin{equation*}
    \wt{\psi}_{\sigma,\vec{\xi}_{j_0}} := \sum_{\vec{k} \in \Z^d} \wt{\psi}_{\sigma,\vec{\xi}_{j_0} + \vec{k}} \,,
  \end{equation*}
  where for its restriction to $[0,1]^d$, of course, only the summation over $\vec{k} \in \{-1,0,1\}^d$ is relevant.
  With $\sigma \leq \frac{1}{2}$, no overlap of shifted copies occurs,
  so we preserve the norm estimate
  $\bigl\|\wt{\psi}_{\sigma,\vec{\xi}_{j_0}}\bigr\|_{W_p^s([0,1]^d)} = \bigl\|\psi_{\sigma,\vec{\xi}_{j_0}}\bigr\|_{W_p^s(\R^d)} \leq 1$,
  as well as the integral value~\eqref{eq:INT(bump)} which decays with $\sigma \to 0$.
  In contrast, for the algorithm $Q_n$ we can state
  that with probability $q(\rho) = q(\sigma/2)$
  the algorithm hits one of the large values of the bump multiplying it with a large weight,
  and with other sample points it detects only zero values outside the support of the bump function,
  in detail:
  \begin{equation*}
    Y_{j_0}(\rho) = Y_{j_0}(\sigma/2) = 1
    \quad\Longrightarrow\quad
      \bigl|Q_n(\wt{\psi}_{\sigma,\vec{\xi}_{j_0}})\bigr|
        \geq t_0 \cdot \frac{b_0}{c_0} \cdot \sigma^{-(d/p - s)} \xrightarrow[\sigma \to 0]{} \infty \,.
  \end{equation*}
  Hence, there exists $0 < \rho_0 \leq \min\{\frac{r_0}{3}, \frac{1}{4}\}$ such that
  $\bigl|Q_n(\wt{\psi}_{\sigma,\vec{\xi}_{j_0}})\bigr| \geq 2 \Int \wt{\psi}_{\sigma,\vec{\xi}_{j_0}}$
  for $\rho = \sigma/2 \in (0, \rho_0]$.
  Define $\delta_0 := q_0 \cdot \rho_0^d$,
  then for $\delta \in (0,\delta_0]$,
  putting $\rho = \sqrt[d]{\delta/q_0}$,
  we find that
  \begin{equation*}
    \left|Q_n(\wt{\psi}_{\sigma,\vec{\xi}_{j_0}}) - \Int \wt{\psi}_{\sigma,\vec{\xi}_{j_0}}\right|
      \geq \frac{t_0 b_0}{2 c_0} \cdot
              \left(\frac{q_0}{2^{d} \delta}\right)^{\frac{1}{p} - \frac{s}{d}} \,.
  \end{equation*}
  with probability greater than $\delta$.
\end{proof}

\begin{remark}
  The theorem above shows that the $\delta$-dependence of the upper bound in Theorem~\ref{thm:s<d/p}
  cannot be improved without choosing different methods for different $\delta$.
  The lower bounds in~\cite[Thm~1]{KR19} imply that
  the $n$ dependence of the error bound in Theorem~\ref{thm:s<d/p} is also optimal,
  that is,
  \begin{equation*}
    e^{\linMC}_{\prob}\left(n,\delta,W_p^s([0,1]^d)\right)
      \asymp n^{-\left(\frac{s}{d} + 1 - \frac{1}{p'}\right)} \cdot c_\delta
  \end{equation*}
  with $p' := \min\{2,p\}$ and $c_\delta > 0$ for $\delta \in (0,\frac{1}{4})$.
  Of course, it would be desirable to have lower bounds that reflect the joint dependence
  of the probabilistic error on $n$ and $\delta$, at least for not too small values of $\delta$.
  The following theorem provides such a lower bound
  but only specifically for the SCV quadrature rule.
\end{remark}

\begin{theorem} \label{thm:SCV-LB}
  Let $s,d \in \N$ and $1 \leq p < \infty$ with $s < d/p$.
  Let $n_0 = n_0(s,d)$ as in~\eqref{eq:n0} and consider
  the respective SCV algorithm $A_{m,s}^{\textup{SCV}}$ for $m \in \N$,
  see~\eqref{eq:SCV}.
  Then there exists a constant $c = c(p,s,d,n_0) > 0$ and $\delta_0 \in (0,1)$
  such that for $\delta \in (0,\delta_0]$ and with $n = 2n_0 \, m^d$ we have
  \begin{equation*}
    e\left(A_{m,s}^{\textup{SCV}}, \delta, W_p^s([0,1]^d)\right)
      \succeq n^{-\left(\frac{s}{d} + 1 - \frac{1}{p}\right)} \cdot \delta^{-\left(\frac{1}{p} - \frac{s}{d}\right)} \,.
  \end{equation*}
\end{theorem}
\begin{proof}
  We reuse some of the notation from the proof of Theorem~\ref{thm:delta-LB}
  and skip details that are parallel.
  Let $\vec{0} := (0,\ldots,0) \in \N_0^d$ denote the zero multi index.
  Consider the bump function $\psi_{\sigma,\vec{x}_0}$
  at $\vec{x}_0 = \left(\frac{1}{8m}, \ldots, \frac{1}{8m}\right)$
  with scaling $0 < \sigma \leq \frac{1}{8m}$.
  With probability $p_1 := 1-2^{-d} \geq \frac{1}{2}$
  the shift parameter satisfies $\vec{\xi} \notin [0,\frac{1}{2}]^d$,
  hence the control variate is zero, in particular $g_{\vec{0},\vec{\xi}} = 0$ on the lower corner sub-cube $G_{\vec{0}}$.
  An individual sample $\vec{X}_{\vec{0}}^{(j)} \sim \Uniform(G_{\vec{0}})$
  hits the ball $B_{\sigma/2}(\vec{x}_0)$ with probability
  \begin{equation*}
    p_2 := \frac{1}{V_d} \cdot \left(\frac{m\sigma}{2}\right)^d \asymp n \cdot \sigma^d \,,
  \end{equation*}
  the probability of any of the sample points $\vec{X}_{\vec{0}}^{(1)},\ldots,\vec{X}_{\vec{0}}^{(n_0)}$
  to hit this ball is
  \begin{equation*}
    p_3 := 1 - (1-p_2)^{n_0}
      \geq \frac{n_0 p_2}{2}
    \qquad\text{for sufficiently small } p_2.
  \end{equation*}
  If any sample point hits the ball, the output of the algorithm is at least
  \begin{equation*}
    A_{m,s}^{\textup{SCV}}(\psi_{\sigma,\vec{x}_0})
      \geq \frac{1}{n_0 \, m^d} \cdot \frac{b_0}{c_0} \cdot \sigma^{-(d/p - s)}
      \asymp n^{-1} \cdot \sigma^{-(d/p - s)}  \,.
  \end{equation*}
  Choosing $\sigma = \frac{\sigma_0}{m} \cdot \delta^{1/d} \asymp (\delta/n)^{1/d}$
  with a suitable constant $0 < \sigma_0 \leq \frac{1}{8}$,
  the ball $B_{\sigma/2}(\vec{x}_0)$ is hit with probability $p_1 p_3 > \delta$
  for small $\delta$ (and, hence, small $p_2$).
  Provided that $\delta$ and thus $\sigma$ is small enough,
  the integral of $\psi_{\sigma,\vec{x}_0}$ is neglectable compared to the output of the algorithm,
  namely
  \begin{equation*}
    e\left(A_{m,s}^{\textup{SCV}}, \delta, W_p^s([0,1]^d)\right)
      \geq \frac{1}{2} A_{m,s}^{\textup{SCV}}(\psi_{\sigma,\vec{x}_0})
      \asymp n^{-\left(\frac{s}{d} + 1 - \frac{1}{p}\right)} \cdot \delta^{-\left(\frac{1}{p} - \frac{s}{d}\right)} \,.
  \end{equation*}
  This is the lower bound as claimed.
\end{proof}

\begin{remark}
  Comparing the upper bounds of Theorem~\ref{thm:s<d/p} with the lower bounds of the theorem above,
  for higher integrability $p > 2$ we see that
  a gap of order $n^{\frac{1}{2} - \frac{1}{p}}$ in the $n$-dependence occurs.
  Considering, say, functions that consist of several bumps on different sub-domains
  does not seem to lead to a larger lower bound.
  This comes as a surprise
  given that general lower bounds for arbitrary (potentially non-linear) algorithms
  are found with functions that fill the domain with bumps
  in the case of high integrability $p \geq 2$, see~\cite[Lem~1]{KR19}.
  The main difference in the proofs is that for general lower bounds
  the source of error is a lack of knowledge on the input function
  while lower bounds for linear methods are caused by the corrupting effect of outliers.
  It remains an open problem whether the upper bound analysis of the SCV algorithm can be improved,
  or whether some trick in the lower bound has the potential of closing the gap.
  
  Yet another much more challenging problem, of course, is
  to prove lower bounds on the joint $(n,\delta)$-dependence for general linear algorithms
  in the low smoothness regime,
  especially if the algorithm can be tuned in to $\delta$.
  One could also try to prove lower bounds for classes of algorithms
  with certain desirable properties.
  For instance, using randomized quasi Monte Carlo methods with equal weights,
  the constant $t_0$ in the proof of Theorem~\ref{thm:delta-LB} can be made explicit.
  If we avoid negative weights, we can lift the restriction
  that only one function evaluation occurs on the support of the bump.
  Algorithms with a high probability that the sampling set has low discrepancy
  (or some other suitable property of ``well distribution'') might allow for more precise estimates
  on the probability of hitting large function values of a bump function.
\end{remark}

\section{Numerical experiments}
\label{sec:numerics}

\begin{figure}[h]
  \includegraphics[width = \linewidth]{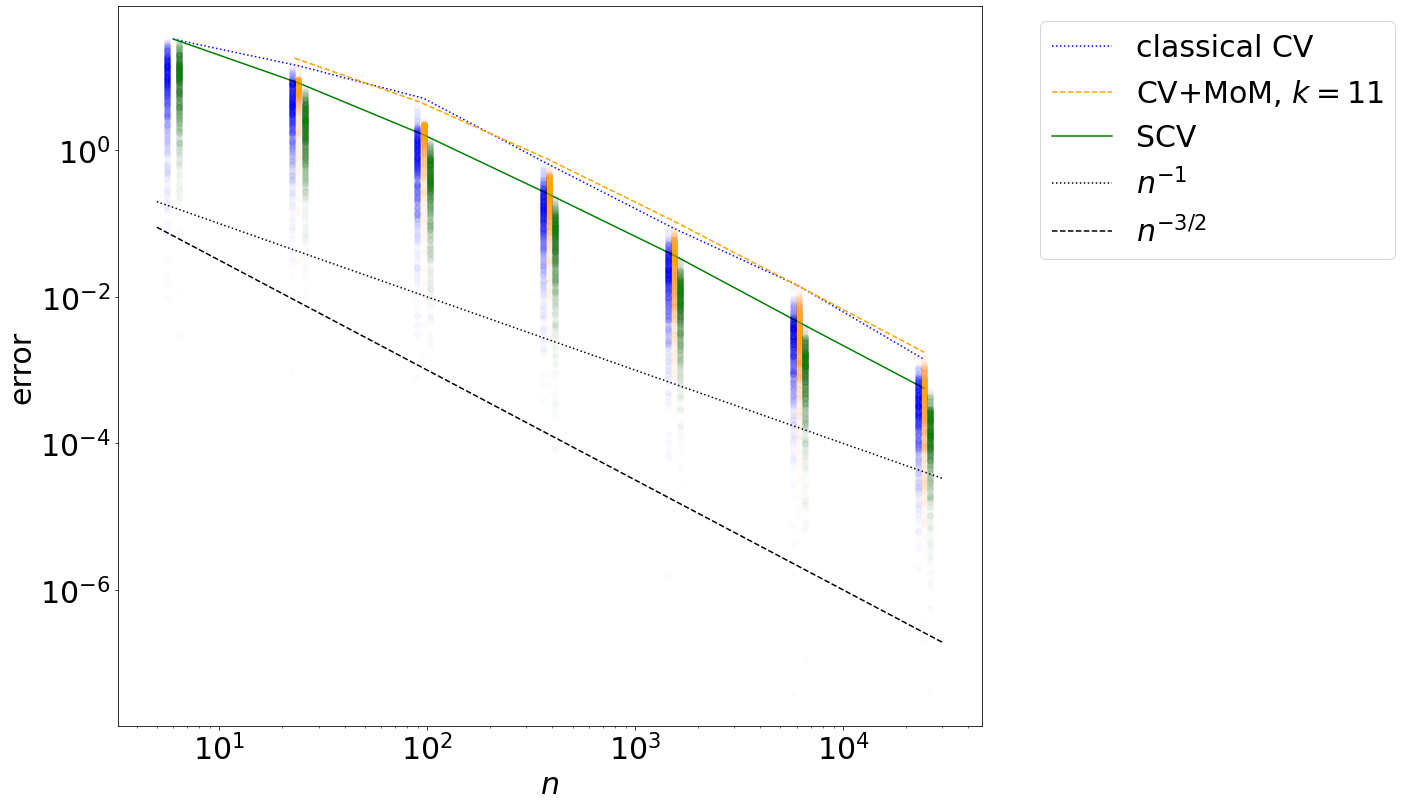}
  \caption{All algorithms were implemented for $s=2$
    and applied to the test function \eqref{eq:test-f}, repeated $1000$ times for each $m=1,2,4,8,16,32,64$.
    Every pale dot represents the observed absolute error for one realization,
    overlapping dots increasing the saturation.
    For the sake of telling the algorithms apart,
    even when they use the same amount $n$ of function values
    for the same subdivision parameter $m$,
    we introduced a slight shift to the left
    for the dots of
    {\color{blue} classical CV}, and a slight shift to the right for {\color{darkgreen} SCV}.
    The top lines mark the largest error we observe among the $1000$ realizations.
    For $s=d=2$, the deterministic worst case error rate is~$n^{-1}$,
    the best possible randomized rate of convergence is $n^{-3/2}$.}
  \label{fig:rate}
\end{figure}

\begin{figure}[h]
  \includegraphics[width = \linewidth]{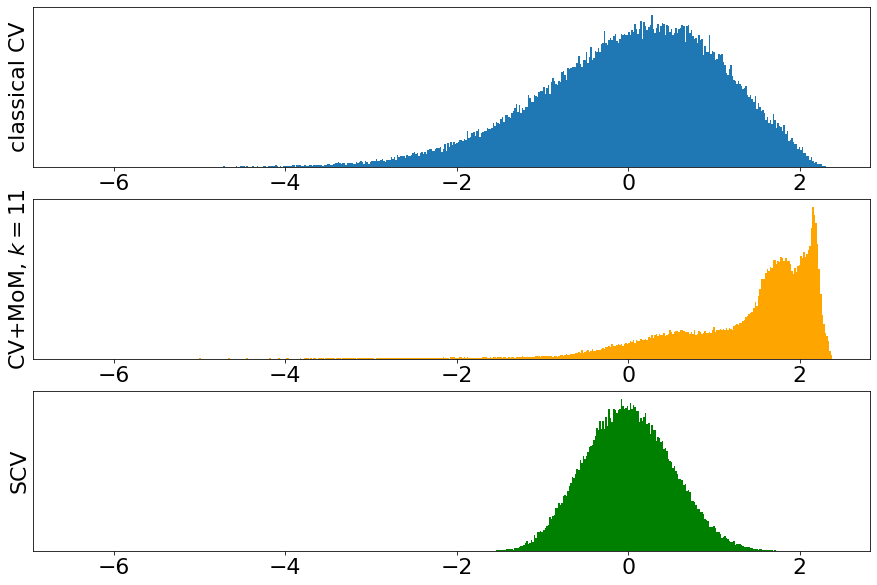}
  \caption{Histograms of the signed error of $10^5$ realizations of each algorithm
    for the test function~\eqref{eq:test-f}
    with algorithmic parameters $s=2$ and $m=4$.
    The range of the error axis is determined by the most extreme single deviations observed.}
  \label{fig:histogram}
\end{figure}

We implemented three versions of control variates
based on the same piecewise polynomial deterministic interpolation on $m^d$ sub-cubes
using \mbox{$n_0 = n_0(s,d)$} interpolation nodes on each sub-cube,
arranged in a simplex regular grid.
Consequently,
all methods work with the same approximation \mbox{$g \colon [0,1]^d \to \R$},
namely, \emph{classical control variates} (CV),
\begin{equation} \label{eq:classicalCV}
  A_{m,s}^{\textup{CV}}(f)
    := \Int g + \frac{1}{n_0 \, m^d} \sum_{i=1}^{n_0 \, m^d} [f-g](\vec{X}_i)\,,
  \qquad \vec{X}_i \iid \Uniform([0,1]^d)\,,
\end{equation}
the \emph{control-variate+median-of-means} approach (CV+MoM),
\begin{equation} \label{eq:CV+MoM}
   A_{m,s,k}^{\textup{CV+MoM}}(f)
     := \Int g + \median\biggl\{ \frac{1}{n_1} \sum_{i=1}^{n_1} [f-g]\bigl(\vec{X}_i^{(j)}\bigr)\biggr\}_{j=1}^k,
     \quad \vec{X}_i^{(j)} \iid \Uniform([0,1]^d)\,,
\end{equation}
where $n_1 := \lfloor n_0 m^d / k \rfloor \geq 1$ for sufficiently large $m$,
and finally the new method \eqref{eq:SCV} we call \emph{stratified control variates} (SCV).
While our implementation is fit for any combination of $s$ and $d$,
the empirical rates of convergence bear no surprises for any of these methods.
Therefore, we focus on studying the error distribution particularly
for $s=2$ (piecewise linear interpolation) and $d=2$ with $n_0 = 3$.
We picked the test function
\begin{equation} \label{eq:test-f}
  f(x_1,x_2) := c \cdot \exp(15x_1 - 5x_2) \,,
\end{equation}
where the constant $c > 0$ is chosen such that $\Int f = 1$.
This function is analytic
but its values and derivatives near the corner $(x_1,x_2) = (1,0)$ are very large,
causing an outlier effect.

As can be observed in Figure~\ref{fig:rate}, for $m=1$ classical CV and SCV
are actually the same method.
In view of the worst out of $1000$ realizations,
it seems like the randomized rate of convergence kicks in earlier for SCV,
whereas for classical CV the decay is almost as slow as for deterministic methods.
CV+MoM could not be implemented for $m=1$ with $k=11$ because we need $n_0 \, m^d \geq k$.
The value $k=11$ is optimized for uncertainty $\delta \approx \frac{1}{100}$, so looking at $1000$ realizations
we should already notice some effect of the median trick,
namely that the error concentrates around a certain value that decays with the desired rate.
A clear difference in the methods can be seen in the histograms of Figure~\ref{fig:histogram}
depicting the signed error $Q_n(f) - \Int f$.
While classical CV is unbiased, we see that the distribution of the error is slightly skewed
since the test function itself is severely skewed.
For CV+MoM we witness a strong bias towards overestimating the integral,
but compared to classical CV we have a better behaviour around the desired confidence level:
Indeed, an absolute error larger than $2.5$ occured for roughly $2\%$ of the classical CV realizations
but only for approximately $1\%$ of the CV+MoM runs,
further, more than $1\%$ of the errors for classical CV exceeded $2.9$.
In contrast, our new SCV quadrature is both unbiased
and highly concentrated around the true solution.

\appendix

\section{Stochastic inequalities}

We start with a Hoeffding type inequality for bounded random variables
where we know a bound on the $\ell_p$-norm of the vector $\vec{b} = (b_i)$
of individual absolute bounds.

\begin{lemma} \label{lem:pnorm-Hoeffding}
  Let $1 < p < 2$ and consider a family $Z_1,\ldots,Z_n$ of
  independent random variables with absolute bounds
  $|Z_i| \leq b_i$, writing $\vec{b} = (b_i)_{i=1}^n \in \R^n$.
  Further, denote the means $a_i := \expect Z_i$
  and their average $a := \frac{1}{n} \sum_{i=1}^n a_i$.
  Then for $\eps > 0$ we have
  \begin{equation*}
    \P\left\{\left|\frac{1}{n} \sum_{i=1}^n Z_i - a\right| > \eps\right\}
      \leq 2\exp\left(-\frac{1}{2} \left(\frac{n \, \eps}{3 \|\vec{b}\|_{\ell_p}}
                \right)^{\frac{p}{p-1}}\right)
      \,.
  \end{equation*}
  Equivalently, for $\delta \in (0,1)$,
  \begin{equation*}
    \left|\frac{1}{n} \sum_{i=1}^n Z_i - a\right|
      \leq 3 \, n^{-1} \cdot \left(2\log\frac{2}{\delta}\right)^{1-1/p} \cdot \|\vec{b}\|_{\ell_p} \,,
  \end{equation*}
  holds with probability at least $1-\delta$.
\end{lemma}
\begin{proof}
  Without loss of generality, $b_1 \geq \ldots \geq b_n \geq 0$.
  Split up the vector into its head $\vec{b}_{[k]} := (b_i)_{i=1}^k$
  and tail $\vec{b}_{\setminus k} := (b_i)_{i=k+1}^n$.
  Using the triangle inequality, we have
  \begin{equation} \label{eq:head+tail}
    \left|\frac{1}{n} \sum_{i=1}^n Z_i - a\right|
      \leq \frac{1}{n} \left(\left|\sum_{i=1}^k (Z_i - a_i)\right|
            + \left|\sum_{i=k+1}^n (Z_i - a_i)\right|\right) \,.
  \end{equation}
  For the head, again via the triangle inequality, we obtain the absolute bound
  \begin{equation} \label{eq:head}
    \left|\sum_{i=1}^k (Z_i - a_i)\right|
      \leq 2\sum_{i=1}^k b_i = 2 \,\|\vec{b}_{[k]}\|_{\ell_1}
      \leq 2 \, k^{1-1/p} \cdot \|\vec{b}_{[k]}\|_{\ell_p} \,.
  \end{equation}
  From Hoeffding's inequality, for the tail and a threshold $t > 0$ we know
  \begin{equation*}
    \P\left\{\left|\sum_{i=k+1}^n (Z_i - a_i)\right| > t\right\}
      \leq 2 \exp\left(- \frac{t^2}{2 \sum_{i=k+1}^n b_i^2}\right) =: \delta \,.
  \end{equation*}
  Resolving for $t$, we find the probabilistic bound
  \begin{equation} \label{eq:tail-2norm}
    t = \sqrt{2\log\frac{2}{\delta}} \cdot \|\vec{b}_{\setminus k}\|_{\ell_2} \,.
  \end{equation}
  We use a well-known fact about best $k$-term approximation,
  \begin{equation} \label{eq:best-k-term}
    \|\vec{b}_{\setminus k}\|_{\ell_2}
      \leq (k+1)^{-\left(\frac{1}{p} - \frac{1}{2}\right)} \cdot \|\vec{b}\|_{\ell_p} \,,
  \end{equation}
  see \cite[Lem~2.1]{Tem86a} for the first occurance of this inequality,
  also check out \cite[eq~(2.4)]{CDdV09} for its application in compressed sensing,
  and \cite[Sec~7.4]{DTU18} for a survey of extensions and historical context.
  Choosing $k := \left\lfloor 2 \log \frac{2}{\delta}\right\rfloor$,
  as long as $2 \log \frac{2}{\delta} < n$, 
  from~\eqref{eq:head}, and \eqref{eq:tail-2norm} combined with~\eqref{eq:best-k-term},
  all plugged into \eqref{eq:head+tail}, we find that
  \begin{equation} \label{eq:boundin_p-norm}
    \left|\frac{1}{n} \sum_{i=1}^n Z_i - a\right|
      \leq \frac{3}{n} \cdot \left( 2 \log \frac{2}{\delta}\right)^{1-1/p}
             \cdot \|\vec{b}\|_{\ell_p} =: \eps
  \end{equation}
  holds with probability $1-\delta$, as claimed.
  If $2 \log \frac{2}{\delta} \geq n$, then we only use the head estimate \eqref{eq:head}
  with $k = n$, which is even better than \eqref{eq:boundin_p-norm}.
  Resolving \eqref{eq:boundin_p-norm} for $\delta$, we obtain
  the alternative representation of the lemma.
\end{proof}

For $L_p$-integrable random variables $Z$, with $1 \leq p < \infty$, we write
\begin{equation*}
  \|Z\|_{L_p} := \left(\expect |Z|^p\right)^{1/p} \,.
\end{equation*}
This norm interpretation helps to obtain the following inequality.

\begin{lemma} \label{lem:RV-triangle-ineq}
  Let $2 \leq q<\infty$ and let
  $Y_1,\ldots,Y_n$ be $L_q$-integrable random variables.
  Then
  \begin{equation*}
    \expect \left(\sum_{i=1}^n |Y_i|^2\right)^{q/2}
      \leq \left(\sum_{i=1}^n \left(\expect |Y_i|^q\right)^{2/q}\right)^{q/2} \,.
  \end{equation*}
\end{lemma}
\begin{proof}
  Writing $p := q/2 \geq 1$, we employ the triangle inequality for the $L_p$-norm and estimate
  \begin{align*}
    \expect\left[\left(\sum_{i=1}^n |Y_i|^2\right)^{q/2}\right]
      &= \left\|\sum_{i=1}^n |Y_i|^2 \right\|_{L_p}^p
      \leq \left(\sum_{i=1}^n \|Y_i^2\|_{L_p}\right)^p
      = \left(\sum_{i=1}^n \|Y_i\|_{L_q}^2\right)^{q/2} \,,
  \end{align*}
  using $\|Y_i^2\|_{L_p} = \left(\expect |Y_i|^{2p}\right)^{1/p} = \left(\expect |Y_i|^{q}\right)^{2/q} = \|Y_i\|_{L_q}^2$. 
\end{proof}

We finally give a Marcinkiewicz-Zygmund type inequality for the $q$-th central absolute moment
of the mean of independent random variables.

\begin{lemma} \label{lem:MZ-ineq}
  For $1 \leq q < \infty$ there exists a constant $c_q > 0$
  such that for any collection $Z_1,\ldots,Z_n$ of independent, $L_q$-integrable random variables with mean
  $a_i := \expect Z_i$, and $a := \frac{1}{n} \sum_{i=1}^{n} a_i$,
  the following inequality holds:
  \begin{equation*}
    \left\|\frac{1}{n} \sum_{i=1}^n Z_i - a\right\|_{L_q}
      \leq \frac{c_q}{n} \cdot 
          \begin{cases}
                \left(\sum\limits_{i=1}^n \|Z_i\|_{L_q}^q \right)^{1/q}
                  &\text{for } 1 \leq q < 2 \\
                \sqrt{\sum\limits_{i=1}^n \|Z_i\|_{L_q}^2 }
                  &\text{for } 2 \leq q < \infty \,.
              \end{cases}
  \end{equation*}
\end{lemma}
\begin{proof}
  For $1 \leq q \leq 2$ we use \cite[Thm~3]{vBEs65}, rescaled by $\frac{1}{n}$, and obtain
  \begin{equation*}
    \left\|\frac{1}{n} \sum_{i=1}^n Z_i - a\right\|_{L_q}
      \leq \frac{2^{1/q}}{n} \left(\sum_{i=1}^n \|Z_i - a_i\|_{L_q}^q\right)^{1/q} \,.
  \end{equation*}

  For $q \geq 2$
  we employ the Marcinkiewicz-Zygmund inequality
  which states that
  \begin{equation} \label{eq:MZ}
    \expect\left(\left|\sum_{i=1}^n (Z_i - a_i)\right|^q \right)
      \leq B_q \expect\left[\left(\sum_{i=1}^n (Z_i - a_i)^2\right)^{q/2}\right] \,,
  \end{equation}
  with $B_q = (q-1)^q$, 
  see~\cite{Burkh88}.
  Lemma~\ref{lem:RV-triangle-ineq} applied to $Y_i := Z_i - a_i$ gives
  \begin{equation*}
    \expect\left[\left(\sum_{i=1}^n (Z_i - a_i)^2\right)^{q/2}\right]
      \leq \left(\sum_{i=1}^n \|Z_i - a_i\|_{L_q}^2\right)^{q/2}.
  \end{equation*}
  Combined with \eqref{eq:MZ}, taking the $q$-th root, and rescaling with $\frac{1}{n}$,
  we get
  \begin{equation*}
    \left\|\frac{1}{n} \sum_{i=1}^n Z_i - a\right\|_{L_q}
      \leq \frac{B_q^{1/q}}{n} \left(\sum_{i=1}^n \|Z_i - a_i\|_{L_q}^2 \right)^{1/2} \,.
  \end{equation*}
  
  Finally, with $\|Z_i - a_i\|_{L_q} \leq \|Z_i\|_{L_q} + |a_i| \leq 2\|Z_i\|_{L_q}$,
  we arrive at the assertion for both cases with the following upper bounds on the constant:
  \begin{equation*}
    c_q \leq \begin{cases}
            2^{1+1/q}
              &\text{for } 1 \leq q \leq 2 \,, \\
            2(q-1)
              &\text{for } 2 \leq q \leq \infty \,.
          \end{cases}
  \end{equation*}
  These estimates on $c_q$ are not optimal. 
  In fact, for $q=1$ one can directly show that even $c_1 \leq 2$ is valid.
  For $q=2$, Bienaym\'e's identity and $\|Z_k - a_i\|_{L_2} \leq \|Z_k\|_{L_2}$
  imply the optimal constant $c_2 = 1$.
\end{proof}

\section*{Acknowledgements}

Several results in this work have been found during the Dagstuhl Seminar 23351
on \emph{Algorithms and Complexity for Continuous Problems},
encouraged by several conversations with colleagues.
I wish to thank the organizers and the Leibniz Center for Informatics
for hosting this event at the inspiring location of Schloss Dagstuhl.

\end{document}